\documentclass[
				letterpaper,	
				10pt,			
				conference 	
] {ieeeconf} 

\IEEEoverridecommandlockouts                              	 

\renewcommand{\baselinestretch}{1} 

\usepackage{graphicx}          		
\usepackage{amsmath}           		
\usepackage{amssymb}           	
\usepackage[noadjust]{cite}      		
\usepackage{subcaption}    
\usepackage{float}
\usepackage{mathrsfs}
\usepackage{dsfont}
\usepackage{xcolor}
\usepackage{comment}

\newcommand{\vect}[1]{\mathbbold{#1}}
\newcommand{\vones}[1][]{\vect{1}_{#1}}
\newcommand{\vzeros}[1][]{\vect{0}_{#1}}
\DeclareSymbolFont{bbold}{U}{bbold}{m}{n}
\DeclareSymbolFontAlphabet{\mathbbold}{bbold}

\usepackage{bbm}
\renewenvironment{proof}{\par\noindent{\textbf{\textit{Proof}}\
}}{\hfill $\square$}

\graphicspath{{./figures/}}

\usepackage{tikz}
\usetikzlibrary{shapes, shapes.geometric, shapes.symbols, shapes.arrows, shapes.multipart, shapes.callouts, shapes.misc}

\newcommand{\setdef}[2]{\{#1 \; : \; #2\}}

\newcommand\oprocendsymbol{\hbox{$\square$}} 
\newcommand\oprocend{\relax\ifmmode\else\unskip\hfill\fi\oprocendsymbol}

\newcommand{\real}{\mathbb{R}}

\renewcommand{\QED}{\QEDopen}

\usepackage{theorem}
\newtheorem{theorem}{Theorem}[section]

\newtheorem{proposition}[theorem]{Proposition}
\newtheorem{lemma}{Lemma}[section]
{\theorembodyfont{\rmfamily} 
\newtheorem{remark}[theorem]{Remark}

\newtheorem{thm:loc_exp_stab}{Theorem}}

\title{On Resistive Networks of Constant Power Devices
\thanks{} 
}

\author{John W. Simpson-Porco, Florian D\"{o}rfler and Francesco Bullo
  \thanks{This work was supported in part by the National Science Foundation NSF CNS-1135819 and by the National Science and Engineering Research Council of Canada. J. W. Simpson-Porco and F. Bullo are with the Center for Control, Dynamical Systems and Computation, University of California at Santa Barbara {\tt \{johnwsimpsonporco,
      bullo\}@engineering.ucsb.edu}. F. D\"orfler is with the Automatic Control Laboratory, Swiss Federal Institute (ETH) Zurich 
{\tt dorfler@ethz.ch}.
Copyright (c) 2014 IEEE. Personal use of this material is permitted.
However, permission to use this material for any other purposes must be
obtained from the IEEE by sending an email to {\tt pubs-permissions@ieee.org}.
}       } 

\begin{document}



\maketitle
\thispagestyle{empty}
\pagestyle{empty}



\begin{abstract}
This brief examines the behavior of DC circuits comprised of resistively interconnected constant power devices, as may arise in DC microgrids containing micro-sources and constant power loads. We derive a sufficient condition for all operating points of the circuit to lie in a desirable set, where the average nodal voltage level is high and nodal voltages are tightly clustered near one another. Our condition has the elegant physical interpretation that the ratio of resistive losses to total injected power should be small compared to a measure of network heterogeneity, as quantified by a ratio of conductance matrix eigenvalues. Perhaps surprisingly, the interplay between the circuit topology, branch conductances and the constant power devices implicitly defines a nominal voltage level for the circuit, despite the explicit absence of voltage-regulated nodes.
\end{abstract}



\section{Introduction}
\label{Section: Introduction}

%
The existence and uniqueness of operating points to nonlinear resistive circuits is a classic topic in circuit theory, with nonlinearities usually entering in the form of voltage or current-controlled nonlinear conductances \cite{LOC-CAD-EDK:87}.
Many elegant approaches have been devised to study such nonlinear circuits, from fixed point and global inverse function theorems to the topological concept of the degree of a mapping \cite{ANW:75}.
%
These approaches offer a binary answer to the question of operating point existence by inferring existence from continuity and monotonicity, or from sector boundedness conditions on the  current-voltage characteristics.
These results offer little guidance in estimating the locations of operating points in voltage-space, or in quantifying their behavior as a function of the circuit topology, branch conductances, and loads.

In this brief we consider a modification of classic nonlinear resistive networks, with two distinguishing features. First, the nonlinearities considered herein arise not from nonlinear branch conductances internal to the circuit, but from externally-connected \emph{constant-power devices} (CPDs).
In contrast to conductance models $i = g(v)$ relating branch-wise voltage and current variables, an externally connected CPD constrains the relationship at the connection port between the port voltage $V$ and the port current injection $I$.
This hyperbolic constraint $I = P/V$ results in a negative (resp. positive) incremental conductance $\mathrm{d}I/\mathrm{d}V$ for sources with $P > 0$ (resp. for loads with $P < 0$), and has been observed to lead to instability in power electronic \cite{AMR-AE:09}, automotive \cite{AD-AK-CHR-GAW:06}, and power transmission systems \cite{IAH:02}. Constant power loads are the most challenging component of the standard ZIP (constant-impedance/current/power) static load model, as their presence often renders analysis and design problems analytically intractable.
Second, as a consequence of the CPDs regulating the power sinked or sourced through each port, the circuit lacks voltage-regulated nodes, and therefore lacks a nominal voltage level around which all nodal voltages cluster. 
For example, in power systems this nominal level is typically supplied by voltage-regulated buses such as generators and points of common coupling \cite{IAH:02,BG-JWSP-FD-SZ-FB:13zb,JWSP-FD-FB:13h}.
%
%
%
%
%
%

{This work is motivated by the increasing deployment of microgrids. These small-footprint power systems offer reliability and flexibility by locally managing generation, storage, and load \cite{JMG-JCV-JM-LGDV-MC:11}.
While microgrids can be directly connected to a utility, their major benefit comes from the ability to disconnect (or ``island'') themselves and operate independently.
Resistive circuits with CPDs arise in islanded DC and AC microgrids consisting of constant-power loads (CPLs) and micro-sources \cite{AK-CNO:11,DM-PM-BNM-SP-BD:12,SS-RO-GB-MM-RG:13}. 
{As an example of a micro-source, photovoltaic panels 
are controlled for maximum power point tracking \cite{CR-GA:07}, and appear to the network as constant sources of power.}
%
We refer the reader to \cite{JMG-JCV-JM-LGDV-MC:11,AK-CNO:11,DM-PM-BNM-SP-BD:12,SS-RO-GB-MM-RG:13} and the references therein for detailed modeling information on micro-sources and CPLs.
Aside from this key application area and circuit-theoretic interest, the proposed analysis and its extensions may prove useful for novel forms of circuit reduction \cite{fd-fb:11d}, synthesis \cite{JWSP-FD-FB:13h}, as a tool for reactive power flow analysis \cite{ECF-LAA-LABT:08,BG-JWSP-FD-SZ-FB:13zb}, and in power transmission networks when generators reach their capability limits \cite{IAH:02}.}

In Section \ref{Sec:Models} we develop the resistive circuit equations, and begin our analysis in Section \ref{Sec:Dec} by proposing a novels decomposition of the vectorized circuit equation.
This decomposition can be thought of as separating the model into two equations: the first describes the resistive losses, while the second describes a complementary ``lossless'' flow of power.
We present a simple and intuitive necessary condition for the existence of operating points. 
{To build intuition for the general analysis which follows,} in Section \ref{Sec:TwoNode} we study in detail the simplest CPD circuit consisting of two ports. We provide a necessary and sufficient condition for the existence of an operating point with positive voltages.
%
%
%
In Section \ref{Sec:Suff} we present our main result, partially generalizing the results of Section \ref{Sec:TwoNode} to arbitrary circuit topologies. We present sufficient parametric conditions for all circuit operating points to belong to an appropriately defined operating region, in which the average nodal voltage is high and all voltages are tightly clustered near one another. 
A loose translation of the condition is the following: the ratio of resistive losses to total transmitted power should be small when compared to the (inverse) network heterogeneity, as quantified by the ``eigenratio'' \cite{ACM-CSZ-JK:05} of conductance matrix. We regard our analysis as a first step in the theoretical understanding of the operating points of CPD circuits. While intuition suggests that the absence of voltage-regulated nodes will result in the circuit displaying a disorganized voltage profile, our main result shows that the voltage profiles of such circuits are quite uniform under appropriate conditions.

The remainder of this section introduces some notation and preliminaries. Given a vector $x \in \real^n$, $\mathrm{diag}(x)$ is the associated diagonal matrix. The $n \times n$ identity matrix is $I_n$, and $\vones[n]$ (resp. $\vzeros[n]$) is the $n$-dimensional vector of all ones (resp. all zeros). The set $\vones[n]^{\perp} \triangleq \setdef{x \in \real^n}{\vones[n]^Tx = 0}$ is the subspace of all vectors in $\real^n$ orthogonal to $\mathrm{span}(\vones[n])$. 
For a vector $x \in \real^n$, $\|x\|_1 = \sum_{k=1}^n |x_k|$, $\|x\|_2 = (\sum_{i=k}^n x_k^2)^{1/2}$ and $\|x\|_{\infty} = \max_k |x_k|$. For a positive semidefinite symmetric $n \times n$ matrix $M$, $\|M\|_2 = \lambda_{\rm max}(M)$ (the largest eigenvalue), while $\|M\|_{\infty} = \max_i \sum_{k=1}^n |M_{ik}|$.

\section{Derivation of Network Equations}
\label{Sec:Models}

\begin{figure}[t!]
\centering
\includegraphics[width=0.8\columnwidth]{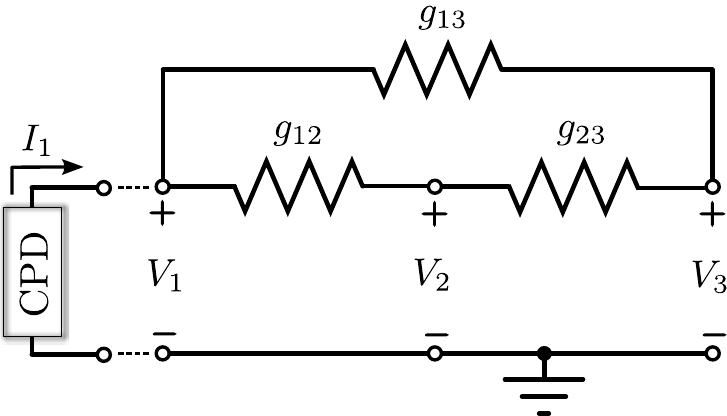}
\caption{{A resistive $n$-port ($n$=3) with constant power devices. By convention, constant power loads satisfy $V_kI_k = P_k < 0$.}}
\label{Fig:Circuit_4}
\vspace{-1em}
\end{figure}

{In this paper we consider linear resistive circuits with nodes $\{0,\ldots,n\}$ for a total of $n+1$ nodes. Nodes $\{1,\ldots,n\}$ are the ones of interest, which in our motivating example of a DC microgrid are locations where power will either be produced or consumed. We assume without loss of generality that these nodes form a connected network. {For the short transmission lines we consider, the $0$th node (ground) is electrically isolated
from the others and taken as the datum.} We may therefore consider the circuit as a ``grounded $n$-port'' where each port voltage is simply the node-to-datum voltage of the respective node \cite{LOC-CAD-EDK:87}. After eliminating any passive internal nodes via Kron reduction \cite{fd-fb:11d}, the input/output behavior of the $n$-port is described in the short-circuit admittance representation by \cite{LOC-CAD-EDK:87}
\begin{equation}\label{Eq:Currr}
I = GV\,,
\end{equation}
where $I = (I_1,\ldots,I_n)$ and $V = (V_1,\ldots,V_n)$ are the vectors of port currents and voltages, and $G$ is the $n \times n$ conductance matrix.
Each port $k$ of our $n$-port is interfaced with a \emph{constant power device}, which constrains the product of the port current $I_k$ and port voltage $V_k$ to be a constant power $P_k$. The power injection $P_k$ is positive for generation and negative for load. An $n$-port with CPDs is shown in Figure \ref{Fig:Circuit_4} for a three-port circuit, where only one CPD has been shown. The port constraints due to CPDs read as
\begin{equation}\label{Eq:Powww}
P = \mathrm{diag}(V)I\,,
\end{equation}
where $P = (P_1,\ldots,P_n)$. Substituting \eqref{Eq:Currr} into \eqref{Eq:Powww}, we arrive at our nonlinear network equations
\begin{equation}\label{Eq:PowerFlow}
\boxed{
P = \mathrm{diag}(V)GV\,.
}
\end{equation}
For future use we collect some useful and well-known facts regarding the conductance matrix $G$ \cite{LOC-CAD-EDK:87}.

\vspace{-1em}

\begin{lemma}[Conductance Matrix]\label{Lem:Cond}
The $n \times n$ conductance matrix $G$ satisfies the following properties:
\begin{enumerate}
\item[(i)] $G = G^T$ is positive semidefinite;
\item[(ii)] $G\vones[n] = \vzeros[n]$;
\item[(iii)] $0 = \lambda_1(G) < \lambda_2(G) \leq \cdots \leq \lambda_n(G)$\,.
\end{enumerate}
\end{lemma}
In contrast with nonsingular conductance matrix models derived from reduced node-edge incidence matrices \cite{LOC-CAD-EDK:87} (i.e., incidence matrices in which the row associated to the datum has been removed), the conductance matrix $G$ in our port model is singular due to the isolation of the datum node. 
This can be seen from Figure \ref{Fig:Circuit_4}: nodes $\{1,2,3\}$ form a cutset, and hence the net current flow through this cutset \textemdash{} or equivalently, the sum over all port currents \textemdash{} must be zero. 
In vector notation, this reads as $\vones[n]^TI = 0$, or using the node equations \eqref{Eq:Currr}, that $\vones[n]^TGV = 0$. This equality holds for all port voltages $V$ if and only if $G\vones[n] = 0\cdot\vones[n]$, which is Lemma \ref{Lem:Cond} (ii). {Equivalently, one may consider $G$ as the conductance matrix of the sub-network containing all nodes other than the datum.}

The second eigenvalue $\lambda_2(G)$ of the conductance matrix is called the \emph{algebraic connectivity} \cite{MF:73,GK-MBH-KEB-MJB-BK-DA:06}, and measures how strongly connected the circuit is. Similarly, the ``eigenratio'' $\lambda_n(G)/\lambda_2(G) \geq 1$ quantifies the \emph{heterogeneity} of the circuit \cite{ACM-CSZ-JK:05}. A large eigenratio $\lambda_n(G)/\lambda_2(G)$ corresponds to a heterogeneous and/or weakly connected circuit, while an eigenratio close to unity corresponds to a highly symmetric and dense circuit with nearly uniform conductances.

An \emph{operating point} is any solution $V$ of the circuit equations \eqref{Eq:PowerFlow}. In contrast to standard circuit problems, the circuit described by \eqref{Eq:PowerFlow} has no voltage-regulated sources -- all port voltages are free variables. The absence of voltage-regulated sources makes determining the operating points of \eqref{Eq:PowerFlow} challenging and non-standard.
Intuitively, one might come to the conclusion that the inflexible power demands of the CPDs would conflict with one another, making equilibrium conditions impossible. In the next three sections we develop some basic theoretical results for the circuit model \eqref{Eq:PowerFlow} and find that this intuition fails.

\section{Decomposition of Circuit Equations}
\label{Sec:Dec}

\begin{figure}[ht!]
\begin{center}
\includegraphics[width=0.65\columnwidth]{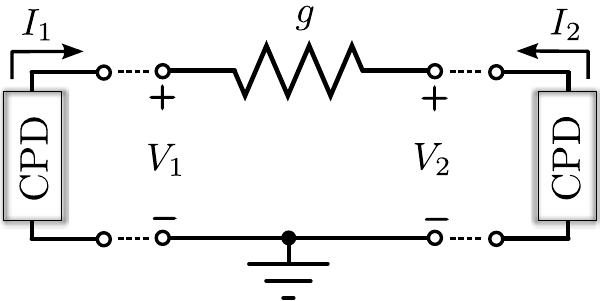}
\caption{{A two-port resistive CPD circuit.}}
\label{Fig:TwoPort}
\end{center}
\vspace{-1em}
\end{figure}

{The starting point for our analysis of \eqref{Eq:PowerFlow} is inspired by the following observations. Under normal conditions, circuits with voltage-regulated nodes (such as power transmission and distribution networks) possess ``almost uniform'' operating points \cite{BG-JWSP-FD-SZ-FB:13zb}, where all nodal voltages are clustered around a nominal value $V_0$. We may express this as
\begin{equation}\label{Eq:Voltage_Decomposition}
V = V_0(\vones[n]+x)\,,
\end{equation}
where $V_0 > 0$ is a nominal voltage level and the vector $x$ is a small, dimensionless deviation variable.\footnote{Without loss of generality, one may assume that $x \in \vones[n]^\perp$.} That is, the voltage profile is the sum of a uniform profile $V_0\vones[n]$, plus a perturbation term described by $x$.} While the solution-space of models such as  \eqref{Eq:PowerFlow} is generally multi-valued and diverse \cite{GSG-AB-AP:08}, we nonetheless are interested in such ``almost uniform'' solutions, as these are the operating points relevant in practice \cite{IAH:02}. Lemma \ref{Lem:Cond} shows that the conductance matrix naturally satisfies a related decomposition, since $G\vones[n] = \vzeros[n]$. Inspired by these properties, we similarly decompose the vector of powers $P = (P_1,\ldots,P_n)$ in \eqref{Eq:PowerFlow} as 
\begin{equation}\label{Eq:Power_Decomposition}
P = \frac{p_{||}}{n}\vones[n] + P_{\perp}\,,
\end{equation}
where $p_{||} \in \real$ and $P_{\perp} \in \vones[n]^\perp$. Such a decomposition is uniquely defined, as one may verify by noting that $p_{||} = \vones[n]^TP = \sum_{k=1}^n P_k$ and that $P_{\perp} = \Pi P$, where $\Pi = I_n - \frac{1}{n}\vones[n]\vones[n]^T$ is the projection matrix onto the subspace $\vones[n]^\perp$. Physically, $p_{||}$ is the total power dissipated in the network (the net difference between sourced and sinked power), while $P_{\perp}$ can be roughly interpreted as the power which flows ``losslessly'' between nodes. Our first result shows that this change of variables allow us to decompose the power balance \eqref{Eq:PowerFlow} into two equations in orthogonal subspaces.


\begin{lemma}[Decomposition] The voltage vector $V=V_0(\vones[n]+x)$ is an operating point of \eqref{Eq:PowerFlow} if and only if
\begin{subequations}
\begin{align}\label{Eq:Ppar}
p_{||} &= V_0^2 x^TGx\,,\\\label{Eq:Pperp}
P_{\perp} &= V_0^2\left(Gx + \mathrm{diag}(x)Gx\right) - \frac{p_{||}}{n}\vones[n]\,.
\end{align}
\end{subequations}
\end{lemma}

\begin{proof}
Substituting \eqref{Eq:Voltage_Decomposition} and \eqref{Eq:Power_Decomposition} into \eqref{Eq:PowerFlow} and using Lemma \ref{Lem:Cond} Property (ii), one obtains
\begin{equation}\label{Eq:Power_Flow_Decomposed}
P_{\perp} + \frac{p_{||}}{n}\vones[n] = V_0^2\left(I_n + \mathrm{diag}(x)\right)Gx\,.
\end{equation}
Equation \eqref{Eq:Ppar} is obtained by left-multiplying \eqref{Eq:Power_Flow_Decomposed} by $\vones[n]^T$, while \eqref{Eq:Pperp} is obtained by left-multiplying \eqref{Eq:Power_Flow_Decomposed} by the projector $\Pi$. The converse direction is immediate.\end{proof}

A \emph{necessary} condition for the existence of an operating point for \eqref{Eq:Ppar}--\eqref{Eq:Pperp} is obtained by noting that the right-hand side of \eqref{Eq:Ppar} is nonnegative. This expresses the fact that the circuit is passive and must dissipate energy, that is, $p_{||} \geq 0$.

\begin{proposition}[Necessary Condition]\label{Prop:Nec}
If $p_{||} < 0$, then \eqref{Eq:Ppar}--\eqref{Eq:Pperp} possesses no operating points. 
\end{proposition}
%
%
%
%
%
%
%
%
%
%
%
%
{Before proceeding in Section \ref{Sec:Suff} to a general analysis of the equations \eqref{Eq:Ppar}--\eqref{Eq:Pperp}, in Section \ref{Sec:TwoNode} we build intuition for the system behavior by studying the simple but illustrative case of a two port circuit (Figure \ref{Fig:TwoPort}).}

\section{Example: Two Port Circuit}
\label{Sec:TwoNode}

The decomposed circuit equations \eqref{Eq:Ppar}--\eqref{Eq:Pperp} can be solved in closed form 
for the specific case of a two port CPD circuit. 
%
%
%
%

\begin{theorem}[Operating Point for Two Node Circuit]\label{Thm:TwoNode}
Consider the circuit \eqref{Eq:PowerFlow} defined for two ports connected by a conductance $g > 0$. Then the circuit  has a high-voltage operating point $V = V_0(\vones[n]+x)$ with a high average voltage $V_0$ and a small percentage deviation $x$ if and only if
\begin{equation}\label{Eq:2nodeParametric2}
\frac{P_1 + P_2}{|P_1 - P_2|} = \frac{p_{||}}{\|P_{\perp}\|_1} \in {]0,1[}\,.
\end{equation}
If this is condition holds, then the unique operating point is\begin{equation}\label{2node:Esoln2}
V = V_0\left(\begin{bmatrix}1 \\ 1\end{bmatrix} + x_{\rm max}\begin{bmatrix}1 \\ -1\end{bmatrix}\right).
\end{equation}
where 
$$
V_0 = \frac{\|P_{\perp}\|_1}{2\sqrt{gp_{||}}} > 0\,,\qquad x_{\rm max} = \frac{p_{||}}{\|P_{\perp}\|_1} < 1\,.
$$
\end{theorem}

%
%
\begin{proof} 
Writing $x = x_{\rm max}\cdot[1;1]$ and $G = g \cdot [1,-1;-1,1]$, 
the decomposition \eqref{Eq:Ppar}--\eqref{Eq:Pperp} reduces to
\begin{align*}
p_{||} &= 4gV_0^2x_{\rm max}^2\,,\\
\|P_{\perp}\|_1 &= 4gV_0^2x_{\rm max}\,,
\end{align*} 
where $p_{||} = P_1 + P_2$ and $\|P_{\perp}\|_1 = \sum_{j=1}^2{|P_{\perp,i}|} = |P_1-P_2|$.
The first equation is simply \eqref{Eq:Ppar}, while the second is obtained by subtracting the two linearly dependent equations in \eqref{Eq:Pperp}. Hence we calculate that $x_{\rm max} = p_{||}/\|P_{\perp}\|_1$ and $V_0 = \pm \|P_{\perp}\|_1/(2\sqrt{gp_{||}})$ and we obtain \eqref{2node:Esoln2}.
%
\hfill \end{proof}

%
%
%
%
\smallskip

Note from the condition \eqref{Eq:2nodeParametric2} that if $P_1 > 0$, then $P_2 < 0$ and vice versa. That is, one node must generate power if the other consumes power. Thus, $\|P_\perp\|_1 = |P_1 - P_2|$ is the \emph{absolute sum} of power injections/demands at the ports of the network, and the parametric condition \eqref{Eq:2nodeParametric2} admits the elegant physical interpretation that the resistive losses in the network $p_{||} = P_1 + P_2$ should be small compared to the gross power transferred through the networks ports (Figure \ref{Fig:TwoPort}).

The dependence of both $x_{\rm max}$ and $V_0$ on the resistive losses $p_{||}$ is intuitive: as $p_{||}$ decreases, the deviation vector $x$ becomes smaller and the average voltage level $V_0$ rises. That is, the voltage profile becomes increasingly uniform. Conversely, as $p_{||}$ increases, $x$ increases in size linearly, while $V_0$ decreases with the square root of the losses. We invite the reader to compare this result with standard load flow results for networks with fixed-voltage buses \cite[Chapter 1]{TVC-CV:98}. Unlike the classic results in \cite{TVC-CV:98}, the existence condition \eqref{Eq:2nodeParametric2} does not depend on the line conductance $g$.

\section{Sufficient Conditions for General Topologies}
\label{Sec:Suff}

{
We now present sufficient conditions on the circuit parameters of a general $n$-port which guarantee that all operating points $V = V_0(\vones[n]+x)$ of \eqref{Eq:PowerFlow} have high average voltage levels $V_0$ and small percentage deviations $x$, generalizing the ``small losses relative to total power'' intuition developed in Section \ref{Sec:TwoNode}. Unlike Theorem \ref{Thm:TwoNode}, we do not provide exact solutions for the circuit operating points, but only bounds for where in voltage-space operating points are located. The results are applicable to general $n$-ports of arbitrary topology.}

\begin{theorem}[Operating Regions I]\label{Thm:2BoundAltGammat} Consider the circuit equations \eqref{Eq:PowerFlow} with the voltage and power decompositions \eqref{Eq:Voltage_Decomposition}--\eqref{Eq:Power_Decomposition} leading to the decomposed circuit equations \eqref{Eq:Ppar}--\eqref{Eq:Pperp}. Assume that the circuit parameters satisfy
\begin{align}\label{Eq:Gamma2Norm}
\Delta &\triangleq \frac{p_{||}}{\|P_{\perp}\|_2-p_{||}}\frac{\lambda_n(G)}{\lambda_2(G)} \, \in \,{]0,\frac{1}{2}[}\,,
\end{align}
and accordingly define 
\begin{subequations}
\begin{align}
	V_{\rm min} &\triangleq \frac{1-\Delta}{\Delta}\sqrt{\frac{p_{||}}{\lambda_2(G)}} > 0\,,
	\label{Eq:2NormVOminAltGamma}\\
		\Bigl.
		x_{\rm max} &\triangleq \frac{\Delta}{1-\Delta} < 1\,.
	\label{Eq:2normboundsAltGamma}
\end{align}
\end{subequations}
Then any operating point $V = V_0(\vones[n]+x)$ of the circuit which exists satisfies $V_0 \geq V_{\rm min}$ and $\|x\|_{\infty} \leq x_{\rm max}$.
\end{theorem}

{Theorem \ref{Thm:2BoundAltGammat} can be understood methodologically as follows. First, using the given data of the problem, compute $\Delta$ from the definition in \eqref{Eq:Gamma2Norm}. If $0 < \Delta < 1/2$, then one may compute $V_{\rm min}$ and $x_{\rm max}$ from \eqref{Eq:2NormVOminAltGamma}--\eqref{Eq:2normboundsAltGamma}, and Theorem \ref{Thm:2BoundAltGammat} guarantees that all operating points $V = V_0(\vones[n]+x)$ of the circuit will have a high average voltage $V_0 \geq V_{\rm min}$ with minimal deviations $\|x\|_{\infty} \leq x_{\rm min}$. 
A graphical interpretation of Theorem \ref{Thm:2BoundAltGammat} for a three-port circuit will be presented later in Section \ref{Sec:Case}.
%
%
%
If $\Delta \leq 0$ or $\Delta \geq 1/2$, we can make no claims regarding the locations of operating points.}

\begin{remark}[\textbf{Interpretation and Qualitative Behavior}]\label{Rem:QualBeh}
The parametric condition \eqref{Eq:Gamma2Norm} implies\footnote{This follows from the fact that $y \leq y/(1-y)$ for $y \in {[0,1[}$.} that $2p_{||}/\|P_{\perp}\|_2 < \lambda_2(G)/\lambda_n(G)$, which may be interpreted as follows: the ratio of resistive losses to total power transmitted must be smaller than the spectral gap of the conductance matrix.
%
%

{It is useful to compare the sufficient condition \eqref{Eq:Gamma2Norm} to the exact two-port condition \eqref{Eq:2nodeParametric2} in Theorem \ref{Thm:TwoNode} by specializing \eqref{Eq:Gamma2Norm} for a two-port network. In this case it is known that $\lambda_2(G) = \lambda_n(G) = g$, and the condition \eqref{Eq:Gamma2Norm} reduces to $p_{||}/\|P_\perp\|_2 < 1/3$. Thus, compared to \eqref{Eq:2nodeParametric2}, the results of Theorem \ref{Thm:2BoundAltGammat} are more conservative due to the $2$-norm and the factor of $1/3$.
}
%
%
%
%
%
Just as was observed in the two-port case of Theorem \ref{Thm:TwoNode}, the variables $x_{\rm max}$ and $V_{\rm min}$ are closely related. In the regime of small losses $p_{||}/\|P_{\perp}\|_2 {<\!\!<} 1$, we observe from \eqref{Eq:Gamma2Norm} that $\Delta \simeq \frac{p_{||}}{\|P_\perp\|_2}\frac{\lambda_n(G)}{\lambda_2(G)} <\!\!< 1$. It follows from \eqref{Eq:2NormVOminAltGamma}--\eqref{Eq:2normboundsAltGamma} that
$$
V_{\rm min} \sim \sqrt{\frac{\lambda_2(G)}{p_{||}}}\,\frac{\|P_{\perp}\|_2}{\lambda_n(G)}\,, \qquad x_{\rm max} \sim \frac{p_{||}}{\|P_{\perp}\|_{2}}\frac{\lambda_n(G)}{\lambda_2(G)}\,.
$$
Thus, the lower bound on the mean voltage level $V_0$ scales with the resistive losses $1/\sqrt{p_{||}}$, while the voltage percentage deviations $x$ are bounded linearly with $p_{||}$. In this regime of low resistive losses, the mean voltage $V_0$ is sensitive to any increase in network losses and will quickly decline, while the percentage deviations $x$ will, at most, increase linearly. These observations lead us to conclude that, much like in standard power systems, an accurate balancing of power supply and power demand is crucial to the efficient and stable operation of CPD-dominated DC microgrids.
\oprocend
\end{remark}

\begin{proof}
Since for any $x \in \vones[n]^\perp$ it holds that $\lambda_2(G)\|x\|_2^2 \leq x^TGx \leq \lambda_n(G)\|x\|_2^2$, from \eqref{Eq:Ppar} we have that
\begin{equation}\label{Eq:2NormUpperBoundOnX2}
\frac{1}{V_0}\sqrt{\frac{p_{||}}{\lambda_n(G)}} \leq \|x\|_2 \leq \frac{1}{V_0}\sqrt{\frac{p_{||}}{\lambda_2(G)}}.
\end{equation}
From \eqref{Eq:Pperp} it holds that
\begin{equation*}
\frac{1}{V_0^2}P_{\perp} = Gx + \mathrm{diag}(x)Gx - \frac{p_{||}}{V_0^2n}\vones[n]\,.
\end{equation*}
Taking norms on both sides and bounding, we obtain
\begin{equation}\label{Eq:Qperpnormed}
\frac{1}{V_0^2}\|P_{\perp}\|_2 \leq \lambda_n(G)(\|x\|_2+\|x\|_2^2) + \frac{1}{V_0^2\sqrt{n}}p_{||}\,
\end{equation}
where we have used the fact that $\|\mathrm{diag}(x)\|_2 = \|x\|_\infty \leq \|x\|_2$. Now using the upper bound in \eqref{Eq:2NormUpperBoundOnX2} and rearranging, \eqref{Eq:Qperpnormed} becomes
\begin{align*}
V_0 &\geq \left(\frac{p_{||}}{\lambda_2(G)}\right)^{1/2}\frac{\lambda_2(G)}{\lambda_n(G)}\left[\frac{\|P_{\perp}\|_2}{p_{||}}-\left(\frac{\lambda_n(G)}{\lambda_2(G)}+\frac{1}{\sqrt{n}}\right)\right]\\
& \geq \left(\frac{p_{||}}{\lambda_2(G)}\right)^{1/2}\frac{\lambda_2(G)}{\lambda_n(G)}\left[\frac{\|P_{\perp}\|_2}{p_{||}}-\left(\frac{\lambda_n(G)}{\lambda_2(G)}+1\right)\right]\\
&= \left(\frac{p_{||}}{\lambda_2(G)}\right)^{1/2}\left[\frac{\lambda_2(G)}{\lambda_n(G)}
\left(\frac{\|P_\perp\|_2}{p_{||}}-1\right)-1\right]\\
&= \left(\frac{p_{||}}{\lambda_2(G)}\right)^{1/2}\left(\frac{1}{\Delta}-1\right) = V_{\rm min}\,.
\end{align*}
We now calculate using \eqref{Eq:2normboundsAltGamma}--\eqref{Eq:2NormVOminAltGamma} and \eqref{Eq:2NormUpperBoundOnX2} that
\begin{align*}
\|x\|_\infty \leq \|x\|_2 &\leq \frac{1}{V_0}\sqrt{\frac{p_{||}}{\lambda_2(G)}} \leq \frac{1}{V_{\rm min}}\sqrt{\frac{p_{||}}{\lambda_2(G)}} = x_{\rm max}\,,
\end{align*}
and hence $\|x\|_{\infty} \leq x_{\rm max}$ as claimed. \hfill
\end{proof}

For completeness we present an analogous $\infty$-norm condition, which similarly restricts the heterogeneity of the network by comparing the smallest branch conductance $g_{\rm min}$ to twice the largest nodal degree $\|G\|_{\infty}$. Depending on the particular topology and heterogeneity of the circuit under consideration, Theorem \ref{Thm:InfBoundAltGammat} may offer more or less conservative results when compared to Theorem \ref{Thm:2BoundAltGammat}.

\begin{theorem}[Operating Regions II]\label{Thm:InfBoundAltGammat} Consider the circuit equations \eqref{Eq:PowerFlow} with the voltage and power decompositions \eqref{Eq:Voltage_Decomposition}--\eqref{Eq:Power_Decomposition} leading to the decomposed circuit equations \eqref{Eq:Ppar}--\eqref{Eq:Pperp}. Assume that the circuit parameters satisfy
\begin{align}\label{Eq:GammaInfNorm}
\widetilde{\Delta} &\triangleq \frac{p_{||}}{\|P_{\perp}\|_{\infty}-p_{||}}\frac{\|G\|_{\infty}}{g_{\rm min}} \, \in \, {]0,1/2[}\,,
\end{align}
where $g_{\rm min} > 0$ is the smallest branch conductance, and accordingly define 
{
$$
\widetilde{V}_{\rm min} \triangleq \frac{1-\widetilde{\Delta}}{\widetilde{\Delta}}\sqrt{\frac{p_{||}}{g_{\rm min}}}\,, \quad \widetilde{x}_{\rm max} \triangleq \frac{\widetilde{\Delta}}{1-\widetilde{\Delta}}\,.
$$
}%
Then all operating points $V = V_0(\vones[n]+x)$ of the circuit satisfy $V_0 \geq \widetilde{V}_{\rm min}$ and $\|x\|_{\infty} \leq \widetilde{x}_{\rm max}$.
\end{theorem}

The proof of Theorem~\ref{Thm:InfBoundAltGammat} is similar to the proof of Theorem~\ref{Thm:2BoundAltGammat} and will not be reported here.

\section{Case Study: Three Node Circuit}
\label{Sec:Case}

To illustrate and test the results presented in Section \ref{Sec:Suff}, we consider the network of Figure \ref{Fig:Circuit_4} with parameters $P_1 = P_3 = -3$kW, $P_2 = 6.6$kW, $g_{12} = g_{23} = 1$S, and $g_{13} = 0.5$S. 
The (in this case, unique) nonlinear solution to \eqref{Eq:PowerFlow} is given by $V = (201.4 \mathrm{V}, 205.2 \mathrm{V}, 223.6 \mathrm{V})$, and hence $V_0 = 210.1$V and $x = (-0.04,-0.02,0.06)$. For these parameters, one calculates using \eqref{Eq:Gamma2Norm} of Theorem \ref{Thm:2BoundAltGammat} that $\Delta = 0.12$. The sufficient condition $0 < \Delta < 1/2$ is satisfied, and one readily calculates from \eqref{Eq:2NormVOminAltGamma} and \eqref{Eq:2normboundsAltGamma} that $V_0 \geq V_{\rm min} = 122V$ and $\|x\|_{\infty} \leq x_{\rm max} = 0.14$. The solution space is depicted in Figure \ref{Fig:example}, which shows that the solution $V$ lies inside the set defined in Theorem \ref{Thm:2BoundAltGammat}, as expected. While the bounds developed in Theorem \ref{Thm:2BoundAltGammat} become increasingly conservative as the number of ports increases, they can be considered as a first step into understanding the fundamental physics of such circuits composed of constant power devices.

\begin{figure}[t]
\begin{center}
\includegraphics[width=0.78\columnwidth]{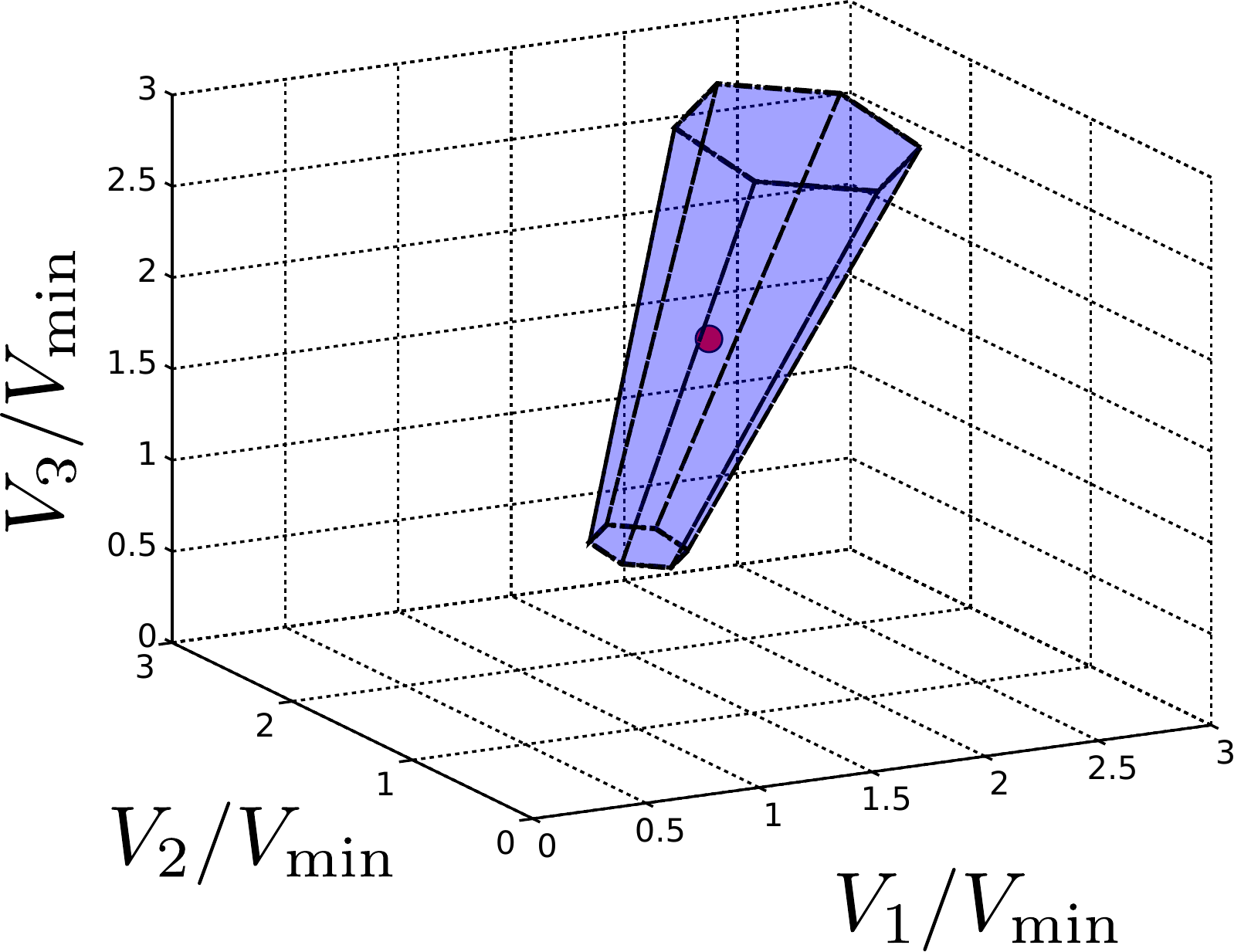}
\caption{Operating point (red) and permitted operating region (blue) for three bus network.}
\label{Fig:example}
\end{center}
\vspace{-1em}
\end{figure}

\section{Conclusions}
{We have examined the behavior of linear resistive circuits with constant power devices interfaced at the circuit ports, and have provided sufficient conditions on the network parameters for circuit operating points to have high average voltage levels and minimal differences in nodal voltage. A curious observation is that despite the absence of voltage-regulated nodes, the network lower bounds its own mean voltage level, as quantified by \eqref{Eq:2NormVOminAltGamma}.}


{While we have partially characterized the circuit operating points through bounds, it remains an open question whether the circuit equations \eqref{Eq:PowerFlow} are exactly solvable for general topologies, and how the results presented change with full ZIP load models. As an outlook to control applications, the relative scaling of the voltage profile heterogeneity $x_{\rm max}$ and the minimum average voltage $V_{\rm min}$ reported in Remark \ref{Rem:QualBeh} hint at strategies for voltage profile control in DC microgrids, in which renewable sources are optimized and dispatched to simultaneously adjust the coupled values $P_{\perp}$ and $p_{||}$.}



\renewcommand{\baselinestretch}{0.98} 

\bibliographystyle{IEEEtran}


\end{document}